\documentclass[11pt]{article}
\usepackage{latexsym,amssymb,amsmath,amsthm,enumerate,geometry,float,cite,tikz}
\newtheorem{theorem}{Theorem}
\newtheorem{lemma}[theorem]{Lemma}

\usepackage{lineno}
\usepackage{setspace}
\usetikzlibrary{snakes}
\usetikzlibrary{calc} 
\allowdisplaybreaks

\begin{document}
\onehalfspace

\title{Mostar index and bounded maximum degree}
\author{
Michael A. Henning$^1$ \and 
Johannes Pardey$^2$\and 
Dieter Rautenbach$^2$\and 
Florian Werner$^2$}
\date{}

\maketitle
\vspace{-10mm}
\begin{center}
{\small 
$^1$ Department of Mathematics and Applied Mathematics, University of Johannesburg,\\ Auckland Park, 2006 South Africa, \texttt{mahenning@uj.ac.za}\\[3mm]
$^2$ Institute of Optimization and Operations Research, Ulm University, Ulm, Germany,\\
\texttt{$\{$johannes.pardey,dieter.rautenbach,florian.werner$\}$@uni-ulm.de}
}
\end{center}

\begin{abstract}
Do\v{s}li\'{c} et al.~defined the Mostar index of a graph $G$ 
as $Mo(G)=\sum\limits_{uv\in E(G)}|n_G(u,v)-n_G(v,u)|$,
where, for an edge $uv$ of $G$,
the term $n_G(u,v)$ denotes the number of vertices of $G$
that have a smaller distance in $G$ to $u$ than to $v$.
For a graph $G$ of order $n$ and maximum degree at most $\Delta$, we show
$Mo(G)\leq 
\frac{\Delta}{2}n^2-(1-o(1))c_{\Delta}n\log(\log(n)),$
where $c_{\Delta}>0$ only depends on $\Delta$
and the $o(1)$ term only depends on $n$.
Furthermore, for integers $n_0$ and $\Delta$ at least $3$, we show
the existence of a $\Delta$-regular graph of order $n$ at least $n_0$ with 
$Mo(G)\geq 
\frac{\Delta}{2}n^2-c'_{\Delta}n\log(n),$
where $c'_{\Delta}>0$ only depends on $\Delta$.\\[3mm]
\end{abstract}

\section{Introduction}

Do\v{s}li\'{c} et al.~\cite{domasktizu} defined the {\it Mostar index} $Mo(G)$ 
of a (finite and simple) graph $G$ as
$$Mo(G)=\sum\limits_{uv\in E(G)}|n_G(u,v)-n_G(v,u)|,$$
where, for an edge $uv$ of $G$,
the term $n_G(u,v)$ denotes the number of vertices of $G$
that have a smaller distance in $G$ to $u$ than to $v$.
A graph is {\it chemical} if it has maximum degree at most $4$.
Since its introduction in 2018 
the Mostar index has already incited a lot of research,
mostly concerning 
sparse graphs and trees  {\cite{AXK, deli1,deli2, HZ1, HZ2, Tepeh}, 
chemical graphs \cite{CLXZ, DL1, GA, GR, XZTHD}, and 
hypercube-related graphs \cite{Mollard, OSS}, see also the recent survey \cite{aldo}.
The largest possible value of the Mostar index of 
graphs, split graphs, and bipartite graphs of a given order
was studied in \cite{gets,mi1,mi2}.

In \cite{domasktizu} Do\v{s}li\'{c} et al.~pose 
the problem to determine 
the chemical graphs as well as the chemical trees 
of a given order with largest Mostar index. 
While Deng and Li \cite{deli1} solve this problem for chemical trees,
little is known about the maximum Mostar index of chemical graphs 
or, more generally, graphs of bounded maximum degree of a given order.
As observed by Sharafdini and R\'{e}ti \cite{sr},
the contribution $|n_G(u,v)-n_G(v,u)|$ of each edge $uv$ 
to the Mostar index of a graph $G$ of order $n$ is at most $n-2$.
For a graph $G$ of maximum degree at most $\Delta$, 
this implies the following trivial bound.
\begin{eqnarray}\label{e1}
Mo(G) & \leq & \frac{\Delta n(n-2)}{2}.
\end{eqnarray}
Addressing the problem from \cite{domasktizu} mentioned above,
we contribute the following results.

\begin{theorem}\label{theorem1}
For integers $n_0$ and $\Delta$ at least $3$, 
there is a $\Delta$-regular graph $G$ of order $n$ at least $n_0$ with
$$Mo(G)\geq \frac{\Delta}{2}n^2-\left(20\Delta^3+12\Delta^2-24\Delta+48\right)n\log_{(\Delta-1)}(n).$$
\end{theorem}

\begin{theorem}\label{theorem2}
For integers $n$ and $\Delta$ at least $3$, 
if $G$ is a graph of order $n$ and maximum degree at most $\Delta$, 
then
$$Mo(G)\leq \frac{\Delta}{2}n^2-(2-o(1))\frac{(\Delta-2)}{(\Delta-1)^2}n\log_{(\Delta-1)}\left(\log_{(\Delta-1)}(n)\right),
$$
where the $o(1)$ term only depends on $n$.
\end{theorem}
For a fixed integer $\Delta$ at least $3$, we conjecture that
the maximum Mostar index of a graph of order $n$ and 
maximum degree at most $\Delta$ is
$\frac{\Delta}{2}n^2-\Theta(n\log(n))$,
where the hidden constants depend on $\Delta$.

The proofs of our theorems and several auxiliary results are given in Section \ref{sec2}.

\section{Auxiliary results and proofs}\label{sec2}

Throughout this section, 
let $\Delta$ be an integer at least $3$.
Since all logarithms have base $\Delta-1$ in this section,
for the sake of brevity, 
we write ``$\log$'' rather than ``$\log_{(\Delta-1)}$''.
Some of our estimates leave room for minor improvements;
as these did not seem essential, we chose to keep the proofs as simple as possible.

For a graph $G$ of order $n$ and an edge $uv$ of $G$, let $\bar{n}_G(v,u)$
denote the number of vertices of $G$ 
whose distance to $v$ is at most their distance to $u$,
that is, we have $\bar{n}_G(v,u)=n-n_G(u,v)$.

Theorem \ref{theorem1} relies on the following lower bound on the Mostar index.

\begin{lemma}\label{lemma-1}
If $G$ is a graph with $n$ vertices and $m$ edges 
and $\vec{G}$ is any orientation of $G$, then
$$Mo(G)\geq n m -2\sum\limits_{(v,u)\in E\left(\vec{G}\right)}\bar{n}_G(v,u).$$
\end{lemma}
\begin{proof}
If $uv$ is an edge of $G$ with $n_G(u,v)\geq n_G(v,u)$, 
then $\bar{n}_G(u,v)\geq \bar{n}_G(v,u)$ and
\begin{eqnarray*}
n-|n_G(u,v)-n_G(v,u)| &=& n-n_G(u,v)+n_G(v,u)\\
&=&\bar{n}_G(v,u)+n_G(v,u)\\
&\leq &2\bar{n}_G(v,u)\\
&=&2\min\{ \bar{n}_G(u,v),\bar{n}_G(v,u)\}.
\end{eqnarray*}
Now,
\begin{eqnarray*}
Mo(G) &=& \sum\limits_{uv\in E(G)}|n_G(u,v)-n_G(v,u)|\\
&=&nm-\sum\limits_{uv\in E(G)}\big(n-|n_G(u,v)-n_G(v,u)|\big)\\
&=&nm-2\sum\limits_{uv\in E(G)}\min\{ \bar{n}_G(u,v),\bar{n}_G(v,u)\}\\
&\geq& n m -2\sum\limits_{(v,u)\in E\left(\vec{G}\right)}\bar{n}_G(v,u),
\end{eqnarray*}
which completes the proof.
\end{proof}

We recall some terminology.

Let $T$ be a rooted tree with root $r$, 
and let $u$ be a vertex of $T$.
The {\it depth} of $u$ in $T$ is the distance between $r$ and $u$ in $T$
and the {\it height} of $u$ in $T$ is the maximum distance between $u$ 
and a leaf of $T$ that is a descendant of $u$.
The {\it height} of $T$ is the height of its root.

Now, let $H$ be an integer at least $2$.
Let $T_H$ be a rooted tree with root $r$ in which 
the root has $\Delta$ children and
every vertex distinct from the root 
either has $\Delta-1$ children
or is a leaf of depth $H$.
There is a natural plane embedding of $T_H$ with $r$ at the top
in such a way that the vertices of $T_H$ of the same depth are linearly ordered, 
say from {\it left} to {\it right}.
All children of some vertex of $T_H$ appear consecutively in this linear order
and, if two vertices $u$ and $v$ of $T_H$ have the same depth 
and $u$ lies left of $v$,
then all children of $u$ lie left of all children of $v$.
If $u$ is not the rightmost vertex of some depth, 
then let $u^+$ denote the vertex of the same depth as $u$
that follows $u$ immediately in the linear order.

Let $u_1,\ldots,u_{\Delta p}$ for $p=(\Delta-1)^{H-1}$ denote the leaves of $T_H$
in their linear order from left to right.
Starting with $T_H$, 
we construct a graph $G_H$ as follows:
For every positive integer $i$ at most $p$,
add a set of $\Delta-1$ new vertices
and add all $\Delta(\Delta-1)$ possible edges between the new vertices
and the vertices in $\{ u_{\Delta(i-1)+1},\ldots,u_{\Delta i}\}$.
We call the vertices and edges of $T_H$ {\it black}
and the vertices and edges of $G_H$ that do not belong to $T_H$ {\it grey}.
Note that $G_H$ is a $\Delta$-regular graph with 
$1+\sum\limits_{i=1}^H\Delta(\Delta-1)^{i-1}$ black vertices
and $(\Delta-1)^H$ grey vertices.
The {\it height} of a black vertex in $G_H$ is its height in $T_H$,
and the neighbors of every grey vertex in $G_H$ are called its {\it parents}. 

See Figure \ref{fig1} for an illustration.

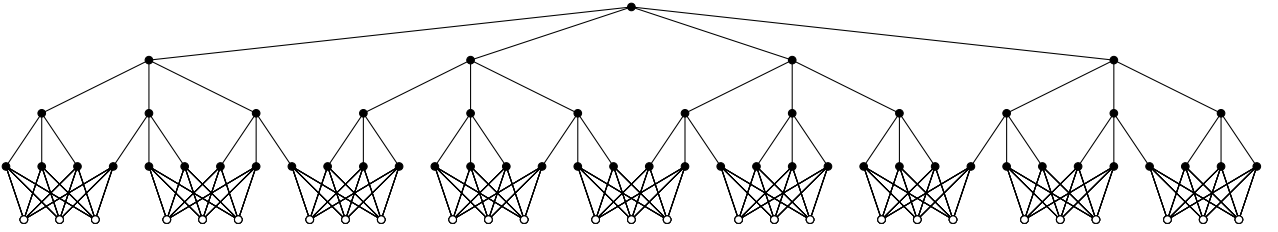
\begin{figure}[H]
	\centering
	\begin{tikzpicture} [scale=0.47]
	\tikzstyle{point}=[draw,circle,inner sep=0.cm, minimum size=1mm, fill=black]
	\tikzstyle{point2}=[draw,circle,inner sep=0.cm, minimum size=1mm, fill=white]

	\xdef\t{2} 
	
	\foreach \i in {-1}{
		\foreach \j in {1}{
			\coordinate (h\j\i) at (2*3^\t,1.5) [label=left:] {};
			\node[point] at (2*3^\t,1.5) [label=left:] {};
		}
	}
	\foreach \k in {0,...,3}{
	\xdef\m{1}
		
	\begin{scope}[shift={(\k*3^\t,0)}]
	\foreach \i in {0,...,\t}{
		\foreach \j in {1,...,\m}{
			\pgfmathparse{3^\t /  3^\i}
			\xdef\jj{\pgfmathresult}
			\coordinate (h\j\i) at (\jj*\j-\jj/2,-1.5*\i) [label=left:] {};
			\node[point] at (\jj*\j-\jj/2,-1.5*\i) [label=left:] {};
			\pgfmathparse{int((\j+2)/3)}
			\xdef\jjj{\pgfmathresult}
			\pgfmathparse{\i-1}
			\xdef\ii{\pgfmathresult}
			\draw (h\j\i) -- (h\jjj\ii);
		}
		\pgfmathparse{3*\m}
		\xdef\m{\pgfmathresult}
	}
\end{scope}
	\pgfmathparse{3^\t-1}
	\xdef\z{\pgfmathresult}
	\foreach \i in {0,...,\z}{
		\begin{scope}[shift={(4*\i+0.5, -1.5*\t)}]
		\foreach \j in {0.5,1.5,2.5}{
			\foreach \i in {0,...,3}{
				\draw (\i,0)--(\j,-1.5);
			}
		\node[point2] at (\j,-1.5) [label=left:] {};
		}
		\end{scope}
}
}

	\end{tikzpicture}
\caption{$G_3$ for $\Delta=4$.}\label{fig1}
\end{figure}

Let $\vec{G}_H$ arise from $G_H$ by orienting all edges of $G_H$ towards the root,
that is, if $u$ is a parent of $v$ in $G_H$, then $E\left(\vec{G}_H\right)$ 
contains the oriented edge $(v,u)$.

\begin{lemma}\label{lemma-2}
If $n$ denotes the order of $G_H$, then
$$\sum\limits_{(v,u)\in E\left(\vec{G}_H\right)}\bar{n}_{G_H}(v,u)
\leq  \left(10\Delta^3+6\Delta^2-12\Delta+24\right)n\log(n).$$
\end{lemma}
\begin{proof}
For a black vertex $v$ of $G_H$ that is not the root $r$,
let $N^{\downarrow}(v)$ denote the set containing 
$v$ as well as all black and grey descendants of $v$.
If $v$ has height $h$, then $N^{\downarrow}(v)$ contains
$\sum\limits_{i=0}^h(\Delta-1)^i\leq (\Delta-1)^{h+1}$ black vertices
and at most 
$\left(\frac{(\Delta-1)^h}{\Delta}+2\right)(\Delta-1)\leq 
(\Delta-1)^h+2(\Delta-1)\leq 3(\Delta-1)^{h+1}$ grey vertices, 
which implies 
\begin{eqnarray}\label{e-1}
\left|N^{\downarrow}(v)\right|\leq 4(\Delta-1)^{h+1}.
\end{eqnarray}
Since there are at most $\Delta(\Delta-1)^{H-h-1}$ black vertices in $G_H$ of height $h$,
we obtain, using (\ref{e-1}),
\begin{eqnarray}\label{e-2}
\sum\limits_{v\in V(T_H)\setminus \{ r\}}\left|N^{\downarrow}(v)\right|
&\leq & \sum\limits_{h=0}^{H-1}\Delta(\Delta-1)^{H-h-1}4(\Delta-1)^{h+1}
=4\Delta(\Delta-1)^HH.
\end{eqnarray}
First, we consider the contribution 
$$\sum\limits_{(v,u)\in E\left(\vec{T}_H\right)}\bar{n}_{G_H}(v,u)$$
of the black edges.
Therefore, let $uv$ be a black edge,
where $v$ is a black vertex of height $h$ and $u$ is the parent of $v$.
Let $N$ be the set of vertices of $G$ 
whose distance to $v$ is at most their distance to $u$,
that is, $\bar{n}_{G_H}(v,u)=|N|$.
Clearly, we have $N^{\downarrow}(v)\subseteq N$.
Removing from $G_H$ all vertices in $N^{\downarrow}(v)$ 
as well as all ancestors of $v$,
the remaining vertices naturally split into two sets $L$ and $R$
that lie to the left and right of the removed vertices, respectively.

See Figure \ref{fig2} for an illustration.

\begin{figure}[H]
	\centering
	\begin{tikzpicture} [scale=0.4]
	\tikzstyle{point}=[draw,circle,inner sep=0.cm, minimum size=1mm, fill=black]
	\tikzstyle{point2}=[draw,circle,inner sep=0.cm, minimum size=1mm, fill=white]
	
\draw [rounded corners, rounded corners=5mm] (-1.12,-2)--(5,8.6)--(11.12,-2)--cycle;

\draw [rounded corners, rounded corners=10mm] (-2,-2)--(4,10.6)--(4,18) -- (-10,-2) --cycle;

\draw [rounded corners, rounded corners=10mm] (12,-2)--(6,10.6)--(6,18) -- (20,-2) --cycle;

\node[point] (v) at (5,7) [label=below:$v$] {};
\node[point] (u) at (4.5,9) [label=right:$u$] {};
\node[point] (w1) at (4.5,11) [label=right:] {};
\node[point] (w2) at (4.7,13) [label=right:] {};
\node[point] (w3) at (5,15) [label=right:] {};
\node[point] (w4) at (5.2,17) [label=above:$r$] {};
	
\draw[very thick] (v) --(u);
\draw (u) -- (w1) --(w2) --(w3)--(w4);	

\draw (w4) -- (2.5,15);
\draw (w4) -- (3,15);
\draw (w4) -- (7,15);

\draw (6.5, 13)-- (w3) -- (3,13);

\draw (2.5, 11)-- (w2) -- (3,11);

\draw (7.5, 9)-- (w1) -- (8,9);

\draw (1.5, 7)-- (u) -- (8,7);


\draw[dashed] (v) -- (-11,7);
\draw[dashed] (18.5,0) -- (-11,0);
\draw[->] (-11,4.2) -- (-11,7);
\draw[->] (-11,2.8) -- (-11,0);
\node at (-11,2.25) [label=above:$h$] {};

\node at (0.5,7.5) [label=above:$L$] {};
\node at (10,7.5) [label=above:$R$] {};
\node at (5,1) [label=above:$N^\downarrow (v)$] {};

\begin{scope}[shift={(1,0)}]
\foreach \j in {0.5,1.5,2.5}{
	\foreach \i in {0,...,3}{
		\draw (\i,0)--(\j,-1.5);
		\node[point] at (\i,0) [label=left:] {};
	}
	\node[point2] at (\j,-1.5) [label=left:] {};
}
\end{scope}
\begin{scope}[shift={(15,0)}]
\foreach \j in {0.5,1.5,2.5}{
	\foreach \i in {0,...,3}{
		\draw (\i,0)--(\j,-1.5);
		\node[point] at (\i,0) [label=left:] {};
	}
	\node[point2] at (\j,-1.5) [label=left:] {};
}
\end{scope}

\begin{scope}[shift={(8.5,0)}]
\foreach \j in {-0.5,0.5,1.5}{
	\foreach \i in {-1,0,3,4}{
		\draw (\i,0)--(\j,-1.5);
		\node[point] at (\i,0) [label=left:] {};
	}
	\node[point2] at (\j,-1.5) [label=left:] {};
}
\end{scope}

\begin{scope}[shift={(-7,0)}]
\foreach \j in {0.5,1.5,2.5}{
	\foreach \i in {0,...,3}{
		\draw (\i,0)--(\j,-1.5);
		\node[point] at (\i,0) [label=left:] {};
	}
	\node[point2] at (\j,-1.5) [label=left:] {};
}
\end{scope}
	\end{tikzpicture}
\caption{The sets $N^{\downarrow}(v)$, $L$, and $R$.}\label{fig2}
\end{figure}
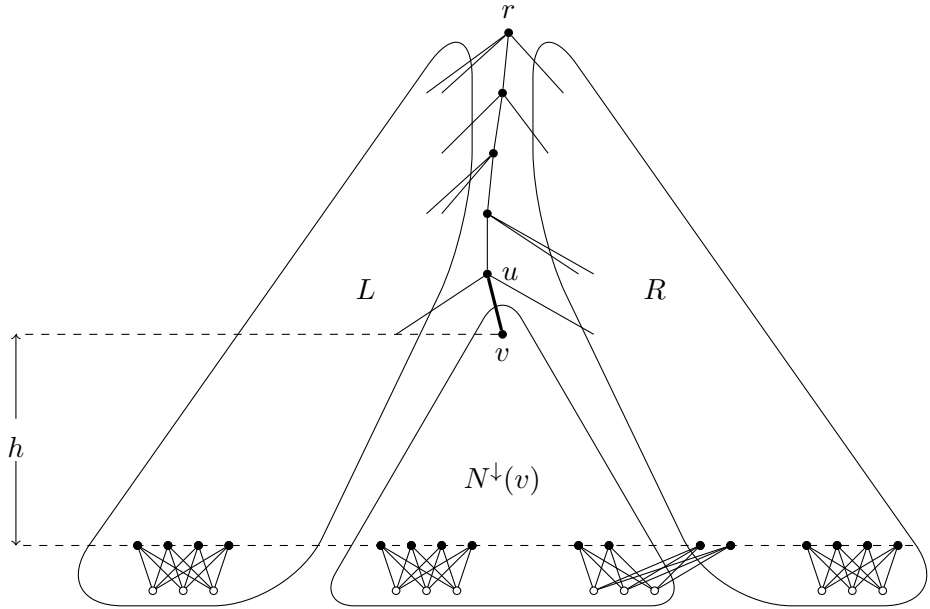

Let 
$N^L(v)=|N\cap L|$
and
$N^R(v)=|N\cap R|$.
Note that $N\setminus (L\cup R)=N^{\downarrow}(v)$.

By the left-right symmetry of $G_H$, we have
\begin{eqnarray}\label{e-3}
\sum\limits_{v\in V(T_H)\setminus \{ r\}}\left|N^L(v)\right|
=\sum\limits_{v\in V(T_H)\setminus \{ r\}}\left|N^R(v)\right|.
\end{eqnarray}
Let $w$ be the rightmost black descendant of depth $H$ of $v$.
Let $s$ be the lowest common ancestor of $w$ and $w^+$,
that is, among all common ancestor of $w$ and $w^+$ 
the vertex $s$ has maximum depth.
If $w$ and $w^+$ have no grey common neighbor,
then $N^R(v)$ is empty.
If $s$ equals $u$, then $w^+$ is closer to $u$ than to $v$
and the structure of $G_H$ implies that $N^R(v)$ is empty.
See Figure \ref{fig3} for an illustration.

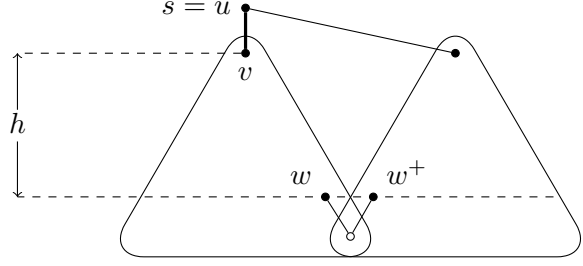
\begin{figure}[H]
	\centering
	\begin{tikzpicture} [scale=0.3] 
	\tikzstyle{point}=[draw,circle,inner sep=0.cm, minimum size=1mm, fill=black]
	\tikzstyle{point2}=[draw,circle,inner sep=0.cm, minimum size=1mm, fill=white]
	
	\draw [rounded corners, rounded corners=5mm] (-1.12,-2)--(5,8.6)--(11.12,-2)--cycle;
	
\begin{scope}[shift={(9.2,0)}]
	\node[point] (vp) at (5,7) [label=below:] {};		
	\draw [rounded corners, rounded corners=5mm] (-1.12,-2)--(5,8.6)--(11.12,-2)--cycle;
\end{scope}

	\node[point] (v) at (5,7) [label=below:$v$] {};
	\node[point] (u) at (5,9) [label=left:\text{$s=u$}] {};
	\draw (u) -- (vp);
	
	\node[point] (w) at (8.5,0.65) [label=above left: $w$] {};
	\node[point2] (b) at (9.6,-1.1) [label=left:] {};
	\node[point] (wp) at (10.6,0.65) [label=above right: $w^+$] {};
	\draw (w) -- (b) --(wp);

	\draw[very thick] (v) --(u);
	
	\draw[dashed] (v) -- (-5,7);
	\draw[dashed] (18.5,0.65) -- (-5,0.65);
	\draw[->] (-5,4.5) -- (-5,7);
	\draw[->] (-5,3.1) -- (-5,0.65);
	\node at (-5,2.55) [label=above:$h$] {};

	\end{tikzpicture}
\caption{In the case $s=u$, the distance between $u$ and $w^+$ is $h+1$ and the distance between $v$ and $w^+$ is $h+2$.
Note that the grey common neighbor of $w$ and $w^+$ may not exist.}\label{fig3}
\end{figure}

Now, suppose that 
$w$ and $w^+$ have a grey common neighbor
and
that $s$ is distinct from $u$.
Let $s$ have height $\ell$,
that is, the distance within $T_H$ between $s$ and $v$ is $\ell-h$.
Note that $\ell\geq h+2$.
Let $v_2$ be the ancestor of $w^+$ of height $\ell-h-2$
and let $u_2$ be the parent of $v_2$.
Note that $v_2$ is closer to $v$ than to $u$
and that $u_2$ is closer to $u$ than to $v$.
This implies $N^{\downarrow}(v_2)\subseteq N^R(v)$.
Let $w_2$ be the rightmost black descendant of depth $H$ of $v_2$.

Before we proceed the estimation, we show,
by an inductive argument over $h_2=\ell-h-2$,
that the distance in 
$G'=G_H\left[N^{\downarrow}(v_2)\right]$ between $w^+$ and $w_2$ is $2h_2$,
in particular, some shortest path in $G'$ 
between $w^+$ and $w_2$ passes through $v_2$.
For $h_2=0$, we have $w^+=w_2=v_2$, and the statement is trivial.
For $h_2>0$, the structure of $G_H$ 
together with the induction hypothesis imply 
that every path in $G'$ between $w^+$ and $w_2$ that avoids $v_2$
has length at least $2(h_2-1)(\Delta-1)+2(\Delta-2)\geq 2h_2$,
which implies the desired statement.

Now, since $u_2$ is an ancestor of $w_2^+$,
the vertex $w_2^+$ is closer to $u_2$ than to $v_2$.
See Figure \ref{fig4} for an illustration.

\begin{figure}[H]
	\centering
	
	\begin{tikzpicture} [scale=0.3]
	\tikzstyle{point}=[draw,circle,inner sep=0.cm, minimum size=1mm, fill=black]
	\tikzstyle{point2}=[draw,circle,inner sep=0.cm, minimum size=1mm, fill=white]
	
	\draw [rounded corners, rounded corners=5mm] (0,-2)--(7.5,11)--(15,-2)--cycle;
	\node[point] (v2) at (7.5,9.3) [label=below:$v_2$] {};	
	\node[point] (u2) at (8.5,11.02) [label=left:$u_2$] {};	
	
	\draw (v2) --(u2);		
	
	\begin{scope}[shift={(11.95,0)}]
	\node[point] (v2p) at (7.5,9.3) [label=below:] {};		
	\draw [rounded corners, rounded corners=5mm] (0,-2)--(7.5,11)--(15,-2)--cycle;
	\end{scope}
	
	\draw (u2) -- (v2p);
	
	\begin{scope}[shift={(-8.05,0)}]
	\node[point] (v) at (5,7) [label=below:$v$] {};
	\node[point] (u) at (4,8.73) [label=left:$u$] {};		
	\draw [rounded corners, rounded corners=5mm] (-1.12,-2)--(5,8.6)--(11.12,-2)--cycle;
	\end{scope}

	\begin{scope}[shift={(3.85,0)}]
	\node[point] (w) at (8.5,0.65) [label=above left: $w_2$] {};
	\node[point2] (b) at (9.6,-1.1) [label=left:] {};
	\node[point] (wp) at (10.6,0.65) [label=above right: $w_2^+$] {};
	\draw (w) -- (b) --(wp);
	\end{scope}
	
	\begin{scope}[shift={(-8.05,0)}]
	\node[point] (w) at (8.5,0.65) [label=above left: $w$] {};
	\node[point2] (b) at (9.6,-1.1) [label=left:] {};
	\node[point] (wp) at (10.6,0.65) [label=above right: $w^+$] {};
	\draw (w) -- (b) --(wp);
	\end{scope}
	
	\draw[very thick] (v) --(u);

	\draw[dashed] (v) -- (-12,7);
	
	\begin{scope}[shift={(-7,0)}]
	\draw[->] (-5,4.5) -- (-5,7);
	\draw[->] (-5,3.1) -- (-5,0.65);
	\node at (-5,2.55) [label=above:$h$] {};
	\end{scope}
	
	\node[point] (x) at ($ (u)+(-4.5,7.8) $) [label=left:] {};
	\draw[snake=snake] (u) to (x);
	
	\node[point] (y) at ($ (u2)+(3.18,5.51) $) [label=left:] {};
	\draw[snake=snake] (u2) to (y);
	
	\node[point] (s) at ($ 0.5*(x)+0.5*(y)+(0,4) $) [label=above:$s$] {};
	
	\draw (x) -- (s) -- (y);

	\coordinate (s2) at (-17,20.53) [label=above:$\ell$] {};
	\draw[dashed] (s2) -- (s);
	
	\draw[->] (-17,11.29) -- (s2);
	\draw[->] (-17,9.89) -- (-17,0.65);
	\node at (-17,9.34) [label=above:$\ell$] {};

	\draw[dashed] (v2) -- (30,9.3);
	\draw[->] (30,5.675) -- (30,9.3);
	\draw[->] (30,4.275) -- (30,0.65);
	\node at (30,3.65) [label=above:$\ell-h-2$] {};
	
	\draw[dashed] (30,0.65) -- (-17,0.65);

	\end{tikzpicture}
\caption{The distance between $u_2$ and $w_2^+$ is $\ell-h-1$ 
and the distance between $v_2$ and $w_2^+$ is $\ell-h$.
Note that the grey common neighbor of $w_2$ and $w_2^+$ may not exist.}\label{fig4}
\end{figure}
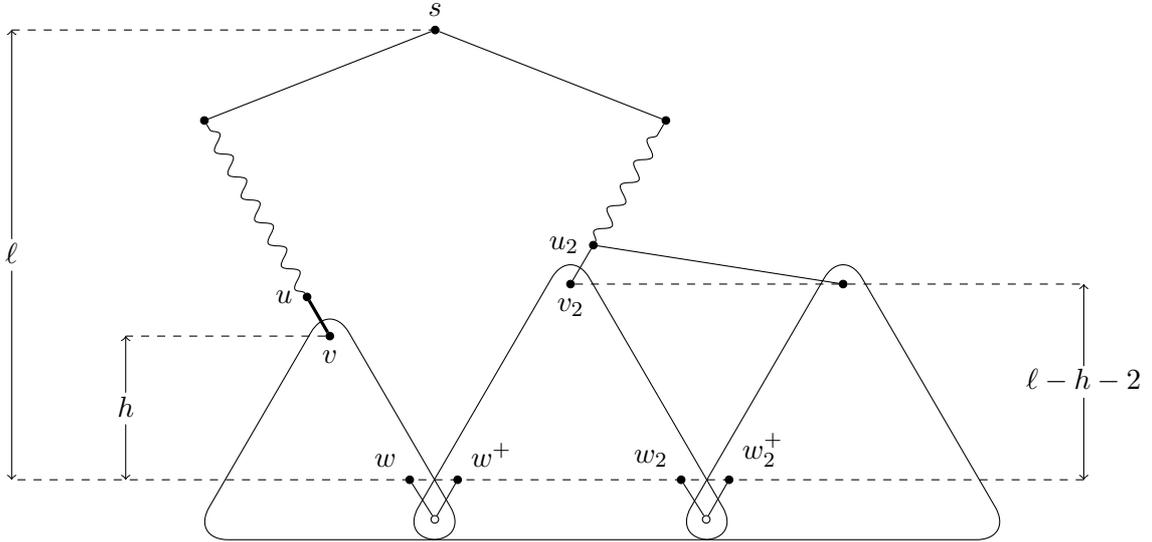

By the above statement about the distance of $w^+$ and $w_2$ in $G'$,
it follows that $w_2^+$ is closer to $u$ than to $v$.
By the structure of $G_H$, 
this implies that $N^{\downarrow}(v_2)=N^R(v)$, and, hence
$$N=N^L(v)\cup N^{\downarrow}(v)\cup N^{\downarrow}(v_2).$$
Altogether, we can estimate the contribution of $v$ 
to the right-hand side of (\ref{e-3}) using $h$ and $\ell$,
and it remains to estimate how many choices of $v$ yield the same values of $h$ and $\ell$.
There are 
\begin{itemize}
\item[(i)] $\Delta(\Delta-1)^{H-\ell-1}$ black vertices of height $\ell$, which are candidates for $s$,
\item[(ii)] at most $(\Delta-1)$ pairs of edges from $s$ to consecutive children 
starting the shortest paths from $s$ to $w$ and $w^+$, and, 
\item[(iii)] by (\ref{e-1}), at most $4(\Delta-1)^{(\ell-h-2)+1}$ vertices in $N^{\downarrow}(v_2)$,
\end{itemize}
which implies
\begin{eqnarray}\nonumber
\sum\limits_{v\in V(T_H)\setminus \{ r\}}\left|N^R(v)\right|
& \leq & 
\sum_{h=0}^{H-2}\sum_{\ell=h+2}^{H}
\underbrace{\Delta(\Delta-1)^{H-\ell-1}}_{(i)}\underbrace{(\Delta-1)}_{(ii)}\underbrace{4(\Delta-1)^{(\ell-h-2)+1}}_{(iii)}\\ \nonumber
& = & 
\sum_{\ell=2}^{H}4\Delta(\Delta-1)^{H-\ell}\sum_{h=0}^{\ell-2}
(\Delta-1)^{\ell-h-1}\\ \nonumber
& \leq  & 
\sum_{\ell=2}^{H}4\Delta(\Delta-1)^{H-\ell}\frac{(\Delta-1)^{\ell}}{(\Delta-2)}\\ \label{e-4}
& \leq & 
\frac{4\Delta}{(\Delta-2)}(\Delta-1)^HH.
\end{eqnarray}
Combining the above estimates, we obtain
\begin{eqnarray}\nonumber
\sum\limits_{(v,u)\in E\left(\vec{T}_H\right)}\bar{n}_{G_H}(v,u)
&=& \sum\limits_{v\in V(T_H)\setminus \{ r\}}
\left(\left|N^L(v)\right|
+\left|N^{\downarrow}(v)\right|
+\left|N^R(v)\right|\right)\\ \nonumber
&\stackrel{(\ref{e-3})}{=}& 2\sum\limits_{v\in V(T_H)\setminus \{ r\}}\left|N^R(v)\right|
+\sum\limits_{v\in V(T_H)\setminus \{ r\}}\left|N^{\downarrow}(v)\right|\\ \nonumber
&\stackrel{(\ref{e-2}),(\ref{e-4})}{\leq} & \frac{8\Delta}{(\Delta-2)}(\Delta-1)^HH
+4\Delta(\Delta-1)^HH\\ \label{e-5}
&\leq & \left(\frac{8\Delta}{(\Delta-2)}+4\Delta\right)(\Delta-1)^HH.
\end{eqnarray}
Next, we consider the contribution 
\begin{eqnarray}\label{e-6}
\sum\limits_{(v,u)\in E\left(\vec{G}_H\right)\setminus E\left(\vec{T}_H\right)}\bar{n}_{G_H}(v,u)
\end{eqnarray}
of the grey edges.
Let $v$ be a grey vertex.
Let $u_1,\ldots,u_\Delta$ be the parents of $v$ in their linear order from left to right.
Since $H\geq 2$, there is some $i\in \{ 1,\ldots,\Delta-1\}$ such that 
$u_1,\ldots,u_i$ have the same parent $p$ and
$u_{i+1},\ldots,u_\Delta$ have the same parent $p^+$.
We call the grey edges between $v$ and $u_1,\ldots,u_i$
the {\it left grey edges}
and the grey edges between $v$ and $u_{i+1},\ldots,u_\Delta$
the {\it right grey edges}.
Similarly as above, by the left-right symmetry of $G_H$, it suffices to consider only the left grey edges.
Now, let $uv$ be a left grey edge incident with $v$.
By symmetry between $u_1,\ldots,u_i$, we may assume that $u=u_i$.
Note that all $i(\Delta-1)\leq (\Delta-1)^2$ grey edges 
between $u_1,\ldots,u_i$ and the set of $\Delta-1$ grey vertices 
with the same neighborhood as $v$
are symmetric and, hence, contribute equally to (\ref{e-6}). 
Again, let $N$ be the set of vertices of $G$ 
whose distance to $v$ is at most their distance to $u$.
Note that $u_1,\ldots,u_{i-1}$ belong to $N$
but that all remaining vertices in $N$ lie to the right
of the path in $G_H$ formed by $v$, $u$, and the ancestors of $u$.
Let $s$ be the lowest common ancestor of $u=u_i$ and $u^+=u_{i+1}$.
Let the height of $s$ be $h+1$.
Since $u$ and $u^+$ have different parents, we have $h\geq 1$.
Let $v_2$ be the ancestor of $u^+$ that is a child of $s$.
Note 
that $s$ is closer to $u$ than to $v$ and 
that $v_2$ is closer to $v$ than to $u$, 
which implies $N^{\downarrow}(v_2)\subseteq N$.
Let $w$ be the rightmost black descendant of depth $H$ of $v_2$.
Let $t$ be the lowest common ancestor of $w$ and $w^+$.
Note that the distance between $u^+$ and $w$ 
in $G_H\left[N^{\downarrow}(v_2)\right]$ is $2h$.

If $t=s$, then $w^+$ is closer to $u$ than to $v$, 
which implies 
$$N=\{ u_1,\ldots,u_{i-1}\}\cup N^{\downarrow}(v_2);$$
see Figure \ref{fig5} for an illustration.

\begin{figure}[H]
	\centering

	\begin{tikzpicture} [scale=0.27] 
	\tikzstyle{point}=[draw,circle,inner sep=0.cm, minimum size=1mm, fill=black]
	\tikzstyle{point2}=[draw,circle,inner sep=0.cm, minimum size=1mm, fill=white]
	
	\draw [rounded corners, rounded corners=5mm] (-1.12,-2)--(5,8.6)--(11.12,-2)--cycle;
	
	\begin{scope}[shift={(9.2,0)}]
	\node[point] (vp) at (5,7) [label=below:] {};		
	\draw [rounded corners, rounded corners=5mm] (-1.12,-2)--(5,8.6)--(11.12,-2)--cycle;
	\end{scope}

	\node[point] (v2) at (5,7) [label=below:$v_2$] {};
		\node[point] (u2) at (0.34,10) [label=above:\text{$t=s$}] {};
	\draw (u2) -- (vp);
	\draw (v2) --(u2);
	
	\node[point] (w) at (8.5,0.65) [label=above left: $w$] {};
	\node[point2] (b) at (9.6,-1.1) [label=left:] {};
	\node[point] (wp) at (10.6,0.65) [label=above right: $w^+$] {};
	\draw (w) -- (b) --(wp);
	
	\draw[dashed] (15,7) -- (-8,7);
	\draw[dashed] (18.5,0.65) -- (-8,0.65);
	\draw[->] (-8,4.5) -- (-8,7);
	\draw[->] (-8,3.1) -- (-8,0.65);
	\node at (-8,2.35) [label=above:$h$] {};

	\node[point2](v) at (0.34,-1.1) [label=right:$v$] {};
\node[point](u) at (-0.66,0.65) [label=above right:] {};
\node at (-1.5,0.4) [label=above right:$u$] {};
\node[point](up) at (1.34,0.65) [label=above right:$u^+$] {};
\draw[very thick] (u) -- (v);
\draw (v) -- (up);
	
	\node[point] (v2m) at (-4.32,7) [label=below:] {};
	\draw[snake=snake] (u) to (v2m);
	
	\draw (v2m) -- (u2);

	\end{tikzpicture}
\caption{In the case $t=s$, 
the distance between $u$ and $w^+$ is $2h+2$ 
and the distance between $v$ and $w^+$ is $2h+3$.
Note that the grey common neighbor of $w$ and $w^+$ 
may not exist.}\label{fig5}
\end{figure}
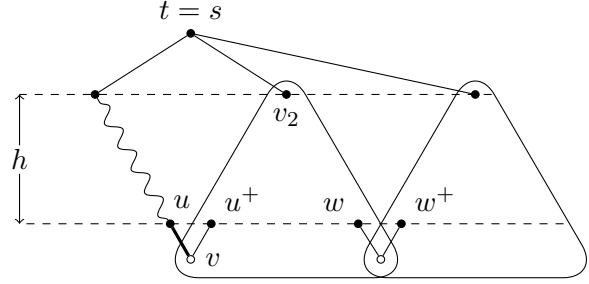

Now, suppose that $t\not=s$.
Let $t$ have height $\ell$.
Note that $\ell\geq h+2$.
Let $v_3$ be the ancestor of $w^+$ of height $\ell-h-2$
and let $u_3$ be the parent of $v_3$.
Note that $v_3$ is closer to $v$ than to $u$ and
that $u_3$ is closer to $u$ than to $v$.
Let $w_2$ be the rightmost black descendant of depth $H$ of $v_3$.
Note that $w_2^+$ is a descendant of $u_3$.
Since the distance of $w^+$ and $w_2$ 
in $G_H\left[N^{\downarrow}(v_3)\right]$ is $2(\ell-h-2)$,
some shortest path between $v$ and $w_2$ passes through $v_3$.
Since $w_2^+$ is closer to $u_3$ than to $v_3$,
it follows that $w_2^+$ is also closer to $u$ than to $v$,
which implies 
$$N=\{ u_1,\ldots,u_{i-1}\}\cup N^{\downarrow}(v_2)\cup N^{\downarrow}(v_3);$$
see Figure \ref{fig6} for an illustration.

\begin{figure}[H]
	\centering
	
	\begin{tikzpicture} [scale=0.27]
	\tikzstyle{point}=[draw,circle,inner sep=0.cm, minimum size=1mm, fill=black]
	\tikzstyle{point2}=[draw,circle,inner sep=0.cm, minimum size=1mm, fill=white]
	
	\draw [rounded corners, rounded corners=5mm] (-1.12,-2)--(5,8.6)--(11.12,-2)--cycle;
	
	\xdef\hhh{13.9} 
	\xdef\www{16} 
	
	\begin{scope}[shift={(16.1,0)}]
	\node[point] (v3) at (0,\hhh-3.6) [label=below:$v_3$] {};
	\node[point] (u3) at (1,\hhh-3.6+1.72) [label=left:$u_3$] {};		
	\draw (v3) -- (u3);		
	\draw [rounded corners, rounded corners=5mm] (-\www/2,-2)--(0,\hhh-2)--(\www/2,-2)--cycle;
	\end{scope}
	
		\begin{scope}[shift={(29.1,0)}]
	\node[point] (v3p) at (0,\hhh-3.6) [label=below:$v_3^+$] {};
	\draw [rounded corners, rounded corners=5mm] (-\www/2,-2)--(0,\hhh-2)--(\www/2,-2)--cycle;
	\end{scope}
		\draw (v3p) -- (u3);		
		
	\begin{scope}[shift={(13,0)}]
			\node[point] (w) at (8.5,0.65) [label=above left: $w_2$] {};
		\node[point2] (b) at (9.6,-1.1) [label=left:] {};
		\node[point] (wp) at (10.6,0.65) [label=above right: $w_2^+$] {};
		\draw (w) -- (b) --(wp);
	\end{scope}

	\node[point] (v2) at (5,7) [label=below:$v_2$] {};
	\node[point] (u2) at ($(v2)+(-1,1.732)$) [label=right:$s$] {};
	\draw (v2) --(u2);
	
	\node[point] (w) at (8.5,0.65) [label=above left: $w$] {};
	\node[point2] (b) at (9.6,-1.1) [label=left:] {};
	\node[point] (wp) at (10.6,0.65) [label=above right: $w^+$] {};
	\draw (w) -- (b) --(wp);
	
	\draw[dashed] (5.7,7) -- (-5,7);
	\draw[->] (-5,4.5) -- (-5,7);
	\draw[->] (-5,3.1) -- (-5,0.65);
	\node at (-5,2.30) [label=above:$h$] {};

  	\node[point2](v) at (0.34,-1.1) [label=right:$v$] {};
	\node[point](u) at (-0.66,0.65) [label=above right:] {};
	\node at (-1.5,0.4) [label=above right:$u$] {};
	\node[point](up) at (1.34,0.65) [label=above right:$u^+$] {};
	\draw[very thick] (u) -- (v);
	\draw (v) -- (up);
	\node[point] (v2m) at (-4.32,7) [label=below:] {};
	\draw[snake=snake] (u) to (v2m);
	
	\draw (v2m) -- (u2);
	
	\node[point] (y) at ($(u3)+	(3.66,6.35)$) [label=below:] {};
	\draw[snake=snake] (u3) to (y);
	
	\coordinate (temp) at (-0.577*\hhh+0.577*0.65+0.577*3.6-0.577*6.3-0.577*1.72,\hhh-0.65-3.6+6.3+1.72) [label=below:] {};;
	\node[point] (x) at ($(w)+(temp)$) [label=below:] {};
	\draw[snake=snake] (u2) to (x);
		
	\node[point] (t) at ($0.5*(x)+0.5*(y)+(0,4)$)	[label=above:$t$] {};
	\draw (x) -- (t)--(y);

	\draw[dashed] ($(v3)+(-0.7,0)$) -- (40,\hhh-3.6);
	\draw[<-] (40,\hhh-3.6) -- (40,0.5*\hhh-1.8+0.65+0.7);
	\draw[<-] (40,0.65) -- (40,0.5*\hhh-1.8+0.65-0.7);
	\node at (40,0.5*\hhh-1.8+0.65-0.7-0.6-0.2) [label=above:$\ell -h-2$] {};
	
	\coordinate (p) at ($(u3)+(17,0)$) [label=above:] {};
	\draw[dashed] (u3) -- (p);
	\draw[<-] (p) -- ($(p)+0.5*(0,6.35)+(0,-0.7)$);
	\coordinate (p2) at ($(p)+(0,6.35)$) [label=above:] {};
	\draw[<-] (p2) -- ($(p)+0.5*(0,6.35)+(0,0.7)$);

	\draw[dashed] (p2) -- (y);
	\node at ($(p)+0.5*(0,6.35)+(0,-0.7)+(0,-0.75)$) [label=above:$h$] {};

	\coordinate (q) at (-10,0) [label=above:] {};
	\coordinate (q2) at (t-|q) [label=above:] {};
	\coordinate (q3) at (u-|q2) [label=above:] {};
	\draw[->] ($0.5*(q2)+0.5*(q3)+(0,0.7)$) -- (q2);
	\draw[->] ($0.5*(q2)+0.5*(q3)+(0,-0.7)$) -- (q3);
	\draw[dashed] (q2) -- (t);
	\node at ($0.5*(q2)+0.5*(q3)+(0,-0.7)+(0,-0.75)$) [label=above:$\ell$] {};
	
	\draw[dashed] (40,0.65) -- (q3);

	\end{tikzpicture}
\caption{In the case $t\not=s$, 
the distance between $u$ and $w_2^+$ is $2\ell$ 
and the distance between $v$ and $w_2^+$ is $1+2h+2+2(\ell-h-2)+2
=2\ell+1$.}\label{fig6}
\end{figure}
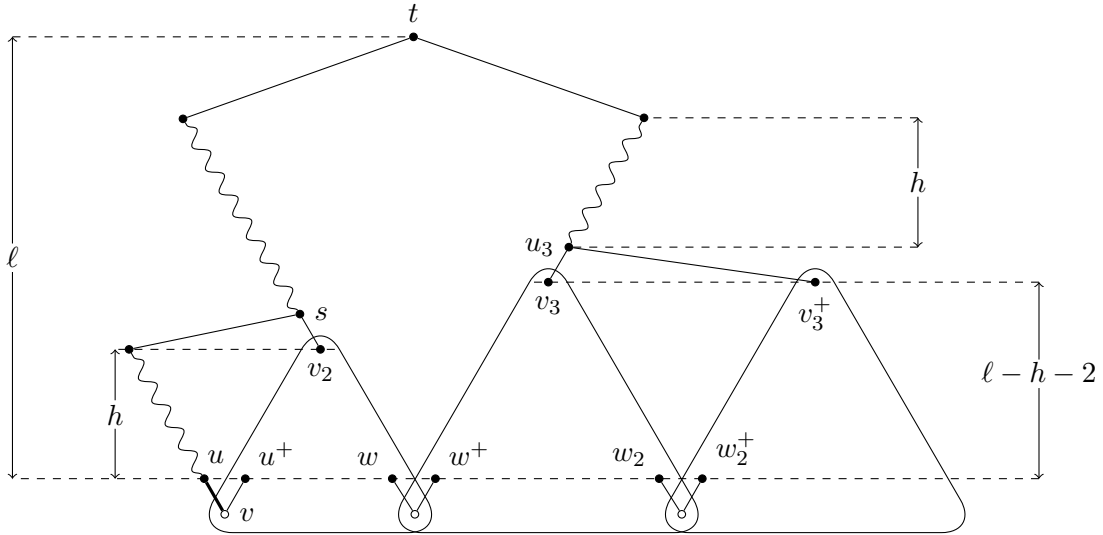

We are now able to estimate the contribution of the left grey edge $uv$ 
to (\ref{e-6}) using $h$ and $\ell$.
As above, we estimate how many choices for the left grey edge $uv$ 
yield the same values of $h$ and $\ell$.

First, we consider the cases in which $\ell=h+1$, that is, $s$ and $t$ coincide.
There are 
\begin{itemize}
\item[(i)] $\Delta(\Delta-1)^{H-(h+1)-1}$ black vertices of height $h+1$, which are candidates for $s$,
\item[(ii)] at most $(\Delta-2)$ triples of edges from $s$ to consecutive children 
starting the shortest paths from $s$ to $u$, $u^+$, $w$, and $w^+$, 
\item[(iii)] at most $(\Delta-1)^2$ left grey edges symmetric to $uv$, 
\item[(iv)] at most $\Delta-2$ vertices in $\{ u_1,\ldots,u_{i-1}\}$, and 
\item[(v)] by (\ref{e-1}), at most $4(\Delta-1)^{h+1}$ vertices in $N^{\downarrow}(v_2)$.
\end{itemize}
Next, we consider the cases in which $\ell\geq h+2$, that is, $s$ and $t$ are distinct.
There are 
\begin{itemize}
\item[(vi)] $\Delta(\Delta-1)^{H-\ell-1}$ black vertices of height $H-\ell$, which are candidates for $t$,
\item[(vii)] at most $(\Delta-1)$ pairs of edges from $t$ to consecutive children 
starting the shortest paths from $t$ to $w$ and $w^+$, 
\item[(viii)] at most $(\Delta-1)^2$ left grey edges symmetric to $uv$, 
\item[(ix)] at most $\Delta-2$ vertices in $\{ u_1,\ldots,u_{i-1}\}$,
\item[(x)] by (\ref{e-1}), at most $4(\Delta-1)^{h+1}$ vertices in $N^{\downarrow}(v_2)$, and
\item[(xi)] by (\ref{e-1}), at most $4(\Delta-1)^{\ell-h-1}$ vertices in $N^{\downarrow}(v_3)$
\end{itemize}
Combining the above estimates, 
and using the symmetry between the left grey edges and the right grey edges, we obtain
\begin{eqnarray}\nonumber
&& \sum\limits_{(v,u)\in E\left(\vec{G}_H\right)\setminus E\left(\vec{T}_H\right)}\bar{n}_{G_H}(v,u)\\
& \leq & \nonumber
2\sum_{h=1}^H
\underbrace{\Delta(\Delta-1)^{H-(h+1)-1}}_{(i)}
\underbrace{(\Delta-2)}_{(ii)}
\underbrace{(\Delta-1)^2}_{(iii)}
\Bigg(\overbrace{\underbrace{(\Delta-2)}_{(iv)}
+\underbrace{4(\Delta-1)^{h+1}}_{(v)}}^{\leq 5(\Delta-1)^{h+1}}\Bigg)\\
&&+2\sum_{\ell=3}^H\sum_{h=1}^{\ell-2}\nonumber
\underbrace{\Delta(\Delta-1)^{H-\ell-1}}_{(vi)}
\underbrace{(\Delta-1)}_{(vii)}
\underbrace{(\Delta-1)^2}_{(viii)}
\Bigg(\overbrace{\underbrace{(\Delta-2)}_{(ix)}
+\underbrace{4(\Delta-1)^{h+1}}_{(x)}}^{\leq 5(\Delta-1)^{h+1}}
+\underbrace{4(\Delta-1)^{\ell-h-1}}_{(xi)}\Bigg)\\
&\leq & \left(10\Delta(\Delta-1)(\Delta-2)+\frac{18\Delta(\Delta-1)^2}{(\Delta-2)}\right)(\Delta-1)^HH.\label{e-7}
\end{eqnarray}
Altogether, using (\ref{e-5}), (\ref{e-7}), and $n\geq (\Delta-1)^H$, 
we obtain
\begin{eqnarray*}
\sum\limits_{(v,u)\in E\left(\vec{G}_H\right)}\bar{n}_{G_H}(v,u)
&\leq &\Bigg(\underbrace{\frac{8\Delta}{(\Delta-2)}}_{\leq 24}+4\Delta+10\Delta(\Delta-1)(\Delta-2)
+\underbrace{\frac{18\Delta(\Delta-1)^2}{(\Delta-2)}}_{\leq 36\Delta(\Delta-1)}\Bigg)\underbrace{(\Delta-1)^HH}_{\leq n\log(n)}\\
&\leq &\left(10\Delta^3+6\Delta^2-12\Delta+24\right)n\log(n),
\end{eqnarray*}
which completes the proof.
\end{proof}
Together Lemma \ref{lemma-1} and Lemma \ref{lemma-2} imply
$$Mo(G_H)\geq \frac{\Delta}{2}n^2-\left(20\Delta^3+12\Delta^2-24\Delta+48\right)n\log(n),$$
which implies Theorem \ref{theorem1}.

\bigskip

We proceed to the proof of Theorem \ref{theorem2}.

We recall some more terminology.
A rooted $(\Delta-1)$-ary tree $T$ of height $h$
is {\it complete} 
if all vertices of $T$ of depth at most $h-1$ have exactly $\Delta-1$ children,
and {\it almost complete} 
if all vertices of $T$ of depth at most $h-2$ have exactly $\Delta-1$ children.
For a vertex $u$ in a rooted tree $T$, 
let $d(u)$ denote the depth of $u$ in $T$ and 
let $n^{\downarrow}(u)$ denote the number of vertices of $T$
that are equal to $u$ or a descendant of $u$.

\begin{lemma}\label{lemma3}
If $T$ is a rooted $(\Delta-1)$-ary tree of order $n$, then
$$\sum\limits_{u\in V(T)}n^{\downarrow}(u)\geq  \frac{(\Delta-2)}{(\Delta-1)^2}n\left(\log\Big((\Delta-2)n\Big)-1\right).$$
\end{lemma}
\begin{proof}
Since each vertex $v$ of $T$ contributes $1$ 
to exactly $d(v)+1$ of the values $n^{\downarrow}(u)$,
double-counting yields 
$$\sum\limits_{u\in V(T)}n^{\downarrow}(u)=\sum\limits_{v\in V(T)}(d(v)+1).$$
Since $T$ is $(\Delta-1)$-ary, 
it is easy to see that $\sum\limits_{v\in V(T)}(d(v)+1)$
is minimized if $T$ is almost balanced.
Hence, for the rest of the proof, we assume that $T$ is almost balanced.

The height $h$ of $T$ satisfies
$\sum\limits_{i=0}^h(\Delta-1)^i=\frac{(\Delta-1)^{h+1}-1}{\Delta-2}\geq n$ and, hence,
$(\Delta-1)^{h+1}\geq (\Delta-2)n$
or equivalently
$h\geq \log((\Delta-2)n)-1$.
Since $T$ contains $(\Delta-1)^{h-1}$ vertices of depth $h-1$, we obtain
\begin{eqnarray*}
\sum\limits_{u\in V(T)}n^{\downarrow}(u) 
& \geq & (\Delta-1)^{h-1}h
\geq \frac{(\Delta-2)}{(\Delta-1)^2}n\left(\log\Big((\Delta-2)n\Big)-1\right),
\end{eqnarray*}
which completes the proof.
\end{proof}
The following lemma contains the key observation for the proof of Theorem \ref{theorem2}.

\begin{lemma}\label{lemma4}
If $T$ is a rooted tree of order $n$ and maximum degree at most $\Delta$, then
$$\sum\limits_{u\in V(T)}\min\left\{ d(u),n^{\downarrow}(u)\right\}
\geq (1-o(1))\frac{(\Delta-2)}{(\Delta-1)^2}n\log(\log(n)),$$
where the $o(1)$ term only depends on $n$.
\end{lemma}
\begin{proof}
Let $d_0=\left\lceil\frac{1}{2}\log(n)\right\rceil$.
Let $A$ be the set of vertices of $T$ of depth less than $d_0$,
let $B$ be the set of vertices $u$ of $T$ of depth at least $d_0$ 
such that $d(u)<n^{\downarrow}(u)$, and 
let $C$ be the set of all remaining vertices of $T$.
Let $a=|A|$ and $b=|B|$.
Since $\Delta\geq 3$, 
we have $a\leq \Delta^{d_0}\leq \Delta\sqrt{n}$.

By definition, all ancestors of every vertex in $B$ lie in $A\cup B$,
and all descendants of every vertex in $C$ lie in $C$.
The root of each component of $T[C]$ has its parent in $A\cup B$.
Since each vertex in $A\cup B$ has at most $\Delta-1$ descendants,
there are $k=(\Delta-1)(a+b)$ 
disjoint possibly empty sets $C_1,\ldots,C_k$
of orders $c_1,\ldots,c_k$
such that the non-empty $C_i$ are the vertex sets of the components of $T[C]$.

Since each component of $T[C]$ is a rooted $(\Delta-1)$-ary tree, 
we obtain, using Lemma \ref{lemma3},
\begin{eqnarray}
\sum\limits_{u\in V(T)}\min\left\{ d(u),n^{\downarrow}(u)\right\}\nonumber
&\geq& \sum\limits_{u\in B}d(u)+\sum\limits_{i=1}^k\sum\limits_{u\in C_i}n^{\downarrow}(u)\\
&\geq& d_0b+\sum\limits_{i=1}^k\frac{(\Delta-2)}{(\Delta-1)^2}c_i\left(\log\Big((\Delta-2)c_i\Big)-1\right).\label{e3}
\end{eqnarray}
If 
$b\geq \frac{n\log(\log(n))}{\log(n)}$,
then (\ref{e3}) implies 
\begin{eqnarray*}
\sum\limits_{u\in V(T)}\min\left\{ d(u),n^{\downarrow}(u)\right\}
&\geq & d_0b
\geq  \frac{1}{2}n\log(\log(n)),
\end{eqnarray*}
which completes the proof.
Hence, we may assume that
$b<\frac{n\log(\log(n))}{\log(n)}$.
Since $a\leq \Delta\sqrt{n}$,
we have 
\begin{eqnarray*}
a+b&=&(1+o(1))\frac{n\log(\log(n))}{\log(n)},
\mbox{ and, hence,}\\
n-(a+b)&=&(1-o(1))n,
\end{eqnarray*}
where the $o(1)$ terms only depend on $n$.

Since the function $f:(0,\infty)\to \mathbb{R}:x\mapsto \gamma_1x(\log (x)-\gamma_2)$ 
for positive $\gamma_1$ and $\gamma_2$ is convex,
the sum in (\ref{e3}), 
considered as a function of the $c_i$ relaxed to non-negative real numbers,
is minimized if all $c_i$ are equal,
that is, 
$$c_i=\frac{|C|}{k}=\frac{n-(a+b)}{k}=\frac{n-(a+b)}{(\Delta-1)(a+b)}=
\frac{(1-o(1))\log(n)}{(\Delta-1)\log(\log(n))}\mbox{ for each $i$.}$$
Now, (\ref{e3}) implies 
\begin{eqnarray*}
\sum\limits_{u\in V(T)}\min\left\{ d(u),n^{\downarrow}(u)\right\}
&\geq& \sum\limits_{i=1}^k\frac{(\Delta-2)}{(\Delta-1)^2}c_i\left(\log\Big((\Delta-2)c_i\Big)-1\right)\\
&\geq& (1-o(1))\frac{(\Delta-2)}{(\Delta-1)^2}n
\left(\log\left(\frac{(\Delta-2)(1-o(1))\log(n)}{(\Delta-1)\log(\log(n))}\right)-1\right)\\
&=& (1-o(1))\frac{(\Delta-2)}{(\Delta-1)^2}n\log(\log(n)),
\end{eqnarray*}
which completes the proof.
\end{proof}
Complete $(\Delta-1)$-ary trees show that the bound in Lemma \ref{lemma4} has the right dependence on $n$.

We are now in a position to complete the proof of Theorem \ref{theorem2}.
Therefore, let $G$ be a graph of order $n$ and maximum degree at most $\Delta$.
Let $K$ be a largest component of $G$.
If $K$ has order less than $n-\log(\log(n))$, then 
$n_G(u,v)\leq n-\log(\log(n))$ for every edge $uv$ of $G$, which implies
$$Mo(G)\leq \frac{\Delta}{2}n(n-\log(\log(n))),$$
and the desired upper bound on $Mo(G)$ follows.
Hence, the order $n(K)$ of $K$ is at least $n-\log(\log(n))$,
that is, 
$$n(K)=(1-o(1))n,$$
where the $o(1)$ term only depends on $n$.

Let $r$ be any vertex of $K$
and let $T$ be a breadth-first search tree of $K$ rooted in $r$.

If $uv$ is an edge of $T$, where $v$ is the parent of $u$,
then $n_G(v,u)\geq d(u)$ and $n_G(u,v)\geq n^{\downarrow}(u)$.

If $n_G(v,u)\geq n_G(u,v)$, then $n_G(v,u)\leq n-n_G(u,v)$,
which implies 
$$|n_G(u,v)-n_G(v,u)|=n_G(v,u)-n_G(u,v)\leq n-2n^{\downarrow}(u),$$
and 
if $n_G(v,u)\leq n_G(u,v)$, then $n_G(u,v)\leq n-n_G(v,u)$,
which implies 
$$|n_G(u,v)-n_G(v,u)|=n_G(u,v)-n_G(v,u)\leq n-2d(u).$$
In both cases, we obtain
\begin{eqnarray}\label{e4}
|n_G(u,v)-n_G(v,u)|\leq n-2\min\left\{ d(u),n^{\downarrow}(u)\right\}.
\end{eqnarray}
Note that $\min\left\{ d(r),n^{\downarrow}(r)\right\}=0$.

Using (\ref{e4}) and Lemma \ref{lemma4}, we obtain
\begin{eqnarray*}
Mo(G) &=& \sum\limits_{uv\in E(G)}|n_G(u,v)-n_G(v,u)|\\
&=& 
\sum\limits_{uv\in E(T)}|n_G(u,v)-n_G(v,u)|
+\sum\limits_{uv\in E(G)\setminus E(T)}|n_G(u,v)-n_G(v,u)|\\
&\leq& \left(m(T)n-2\sum\limits_{u\in V(T)}\min\left\{ d(u),n^{\downarrow}(u)\right\}\right)
+(m(G)-m(T))n\\
&\leq & m(G)n-(2-o(1))\frac{(\Delta-2)}{(\Delta-1)^2}n(K)\log(\log(n(K)))\\
&\leq & \frac{\Delta}{2}n^2-(2-o(1))\frac{(\Delta-2)}{(\Delta-1)^2}n\log(\log(n)),
\end{eqnarray*}
which completes the proof of Theorem \ref{theorem2}.

\bigskip

\noindent {\bf Acknowledgement.}
The research reported in this paper was carried out at the 
{\it 26th C5 Graph Theory Workshop} 
organized by Ingo Schiermeyer in Rathen, May 2023.
We express our gratitude to Ingo, 
not just for this year but also for the long tradition of this wonderful meeting.

\end{document}